\documentclass[12pt]{article}
\usepackage{amsmath}
\usepackage{amssymb}
\usepackage{amsfonts}
\usepackage{amsthm}
\usepackage{ascmac}
\usepackage{authblk}
\usepackage{indentfirst}
\usepackage{graphicx}  
\usepackage{amscd}

\title{A Note on Tamely Ramified Class Field Theory of Two Dimensional Local Rings}
\author{Shinji Ishida\thanks{Email address: nanacanji@gmail.com}}

\providecommand{\keywords}[1]{{\textit{Keywords:}} #1}

\makeatletter
\@addtoreset{equation}{section}

\newtheorem{defi}{Definition}[section]
\newtheorem{thm}{Theorem}[section]
\newtheorem{prop}{Proposition}[section]

\begin{document}
\maketitle

\begin{abstract}
In this note, we treat two dimensional complete local rings which are the completion of $\mathcal{O}_{K}[x, y] = \mathcal{O}_{K}[X, Y]/(XY-\pi)$ at a maximal ideal $(x, y)$, where $\mathcal{O}_{K}$ is an integer ring of a finite extension field $K$ of $\mathbb{Q}_{p}$ or $\mathbb{F}_{p}((t))$ which contains $\pi^{1/m}$, and $\pi=p$ if $\mathbb{Q}_{p}\subset \mathcal{O}_{K}$, $\pi=t$ if $\mathbb{F}_{p}((t))\subset \mathcal{O}_{K}$, and we assume that  the order of the group of root of unity of $\mathcal{O}_{K}$ is $m$ and $m$ is not divisible by $p$. We discuss the tamely ramified class field theory of the fractional field $K_{R}$ of R. This class field theory allows $x$ and $y$ to ramify tamely and it is unramified at height 1 prime ideals of $R$ other than $x$ and $y$.
\end{abstract}
\keywords\footnote[1]{The subject classification codes by 2020 Mathematics Subject Classification is primary-19F05.}{Generalized class field theory (K-theoretic aspects) }



\section{Introduction}
We observe tamely ramified class field theory of two dimensional complete local rings which is a completion of $\mathcal{O}_{K}[x, y] = \mathcal{O}_{K}[X, Y]/(XY-\pi)$ at a maximal ideal $(x, y)$, where $\mathcal{O}_{K}$ is an integer ring of a finite extension field $K$ of $\mathbb{Q}_{p}$ or $\mathbb{F}_{p}((t))$ which contains $\pi^{1/m}$, and $\pi=p$ if $\mathbb{Q}_{p}\subset \mathcal{O}_{K}$, $\pi=t$ if $\mathbb{F}_{p}((t))\subset \mathcal{O}_{K}$. Further we assume that  the order of the group of root of unity of $\mathcal{O}_{K}$ is $m$.\\

In \cite{Saito}, Professor S. Saito completed the unramified class field theory of two dimensional local rings, and it is essential. In this note, our tamely ramified class field theory allows $x$ and $y$ to ramify tamely (the ramification degrees are prime to p) and it is unramified at height 1 prime ideals of $R$ other than $x$ and $y$. Since the unramified part was already completed in \cite{Saito}, we discuss the tame ramification and this was a theme of the author's master thesis 25 years ago. His academic supervisor shared a great idea "Logarithmic Chow Group" around 1997 with him for the author's master thesis. Approximately 25 years have past after that because of the author's procrastination. The author would like to put the great idea out into the world to express his heartflet thanks to his supervisor.\\


\section{Definition of $SK_{1}^{log}(R)$ and the fundamental property}
Let $R$ be a two dimensional local ring which is a completion of $\mathcal{O}_{K}[x, y] = \mathcal{O}_{K}[X, Y]/(XY-\pi)$ at a maximal ideal $(x, y)$, where $\mathcal{O}_{K}$ is an integer ring of a finite extension field $K$ of $\mathbb{Q}_{p}$ or $\mathbb{F}_{p}((t))$ which contains $\pi^{1/m}$, and $\pi=p$ if $\mathbb{Q}_{p}\subset \mathcal{O}_{K}$, $\pi=t$ if $\mathbb{F}_{p}((t))\subset \mathcal{O}_{K}$. Further we assume that the order of the group of root of unity of $\mathcal{O}_{K}$ is $m$ and it is not divisible by $p$. Let $K_{R}$ be the fractional field of $R$. Then the definition of $SK_{1}^{log}(R)$ as as follow:\\

\begin{defi}[$SK_{1}^{log}$]
$SK_{1}^{log}(R)$ is the cokernal of

\begin{equation}
K_{2}^{M}K_{R} \to (\bigoplus_{\mathfrak{p}\in P} k(\mathfrak{p})^{\times}) \bigoplus (K_{2}^{M}\hat{K}_{R,x}/U^{1}K_{2}^{M}\hat{K}_{R,x}) \bigoplus (K_{2}^{M}\hat{K}_{R,y}/U^{1}K_{2}^{M}\hat{K}_{R,y})
\end{equation}
\end{defi}
where $P$ is a set of all prime ideals of height 1 in $R$ other than $(x)$ and $(y)$, $k(\mathfrak{p})$ is the residue field of $R$ at $\mathfrak{p}$, $K_{2}^{M}$ means Milnor K-group of field, $U^{1}K_{2}^{M}K=\{1+\mathfrak{m}_{K}\mathcal{O}_{K},*\}$ for a local filed $K$ ($\mathfrak{m}_{K}$ is the maximal ideal of an integer ring $\mathcal{O}_{K}$ of $K$), and $\hat{K}_{R,x}$ (resp.$\hat{K}_{R,y}$) is fractional field of a completion of $R$ at $(x)$ (resp.at $(y)$).

For $\mathfrak{p}\in P$, $K_{2}^{M}K_{R} \to k(\mathfrak{p})^{\times}$ is the usual tame symbol:
\begin{equation}
\{f,g\}\to (-1)^{ord_{\mathfrak{p}}(f)ord_{\mathfrak{p}}(g)}f^{ord_{\mathfrak{p}}(g)}g^{-ord_{\mathfrak{p}}(f)}|_{\mathfrak{p}}.
\end{equation}

For $x$, $y$, the following homomorphism is induced from the natural inclusion $K_{R}\subset \hat{K}_{R,x}$, $\hat{K}_{R,y}$:
\begin{equation}
K_{2}^{M}K_{R} \to (K_{2}^{M}\hat{K}_{R,x}/U^{1}K_{2}^{M}\hat{K}_{R,x}) \bigoplus (K_{2}^{M}\hat{K}_{R,y}/U^{1}K_{2}^{M}\hat{K}_{R,y}).
\end{equation}

The fundamental property of $SK_{1}^{log}(R)$ is as follow:

\begin{prop}
Let $F$ be the finite residue field of $R$.  Then the kernel $Ker(\partial)$ of the following natural surjective $\partial$ has the following group structure:

\begin{equation}
SK_{1}^{log}(R)  \xrightarrow{\partial}  SK_{1}(R).
\end{equation}
\begin{equation}
0 \longrightarrow D \longrightarrow Ker(\partial) \longrightarrow F^{\times}  \longrightarrow 1
\end{equation}
where $D$ is a divisible group.
\end{prop}

\begin{proof}
First, for $x$ (resp. $y$), the kernel $U^{0}K_{2}^{M}\hat{K}_{R,x}$ of the tame symbol $K_{2}^{M}\hat{K}_{R,x}\to$ $k(x)^{\times}$ contains $U^{1}K_{2}^{M}\hat{K}_{R,x}$. Therefore this induces the surjectivity of $\partial$. Furthermore, it is well known that $U^{0}K_{2}^{M}\hat{K}_{R,*}/U^{1}K_{2}^{M}\hat{K}_{R,*}\simeq K_{2}^{M}k(*)$ for $*=$ $x$ and $y$, where $k(x)=F((x))$ and $k(y)=F((y))$. Therefore, we have the following exact sequence:
\[
\begin{CD}
K_{2}^{M}k(x)\oplus K_{2}^{M}k(y) @>{d}>>  SK_{1}^{log}(R)  @>>>  SK_{1}(R)  @>>>  0. \\
\end{CD}
\]
We investigate the image of the homomorphism $d$.\\

For $\{u,x\}\in K_{2}^{M}K_{R}$ and $u\in K^{\times}$, the image of (2.3) in is $\{u|_{\pi},x\}\oplus u|_{\pi}\in K_{2}^{M}k(x)\oplus k(y)^{\times}$. Similarly, for $\{v,y\}\in K_{2}^{M}K_{R}$ and $v\in K^{\times}$, the image of (2.3) in is $\{v|_{\pi},x\}\oplus v|_{\pi}\in K_{2}^{M}k(y)\oplus k(x)^{\times}$. Note that $K_{2}^{M}k(x)$ (resp. $K_{2}^{M}k(y)$) is a sum of $F^{\times}$ and a divisible group. This indicates that the image of $\{u,xy\}\in K_{2}^{M}K_{R}$ ($u\in K^{\times}$) via (2.3) is $F^{\times}$ and it is the diagonal image of $F^{\times}$ via $d$. Therefore, we have the following group structure:
\begin{equation}
0 \longrightarrow D \longrightarrow Ker(\partial) \longrightarrow F^{\times}  \longrightarrow 1
\end{equation}
where $D$ is a divisible group, and note that the quotient of the divisible group is also a divisible group.
\end{proof}


\section{Tame abelian Galois Group and Reciprocity Map}
We recall reciprocity map of tamely ramified extension.\\
Let $K$ be a discrete valuation field with a uniformiser $t$ and the residue field $F$, $\mu_{n}$ be the group of root of unity of $K$. We assume that $n$ is prime to characteristics ch($K$) and ch($F$) of $K$ and $F$. Then we have the canonical isomorphism between Galois Group of tamely ramified Galois extension $K(t^{\frac{1}{n}})/K$ and $\mu_{n}$:\\

$$
\begin{array}{ccc}
Gal(K(t^{\frac{1}{n}})/K) & \stackrel{f}{\longrightarrow} &  \mu_{n} \\
\rotatebox{90}{$\in$} & & \rotatebox{90}{$\in$} \\
\sigma & \longmapsto & \frac{\sigma(t^{\frac{1}{n}})}{t^{\frac{1}{n}}}
\end{array}
$$

Although an equation $X^{n}=t$ over $K$ has distinct roots $\omega^{r}t^{\frac{1}{n}}$ ($0\leqq r \leqq n-1$ and $\omega$ is a primitive root of unity $\in \mu_{n}$), the isomorphism $f$ is independent of the choice of $\omega^{r}t^{\frac{1}{n}}$. Actually, since $K$ acts $\omega$ trivially, 

\begin{equation}
\frac{\sigma(\omega^{r} t^{\frac{1}{n}})}{\omega^{r} t^{\frac{1}{n}}} = \frac{\omega^{r} \sigma(t^{\frac{1}{n}})}{\omega^{r} t^{\frac{1}{n}}}=\frac{\sigma(t^{\frac{1}{n}})}{t^{\frac{1}{n}}}
\end{equation}
Furthermore, if $\sigma$ and $\tau$ $\in Gal(K(t^{\frac{1}{n}})/K)$,
\begin{equation}
f(\sigma \tau)=\frac{\sigma \tau (t^{\frac{1}{n}})}{t^{\frac{1}{n}}} = \frac{\sigma (\tau (t^{\frac{1}{n}}))}{\tau (t^{\frac{1}{n}})} \frac{\tau (t^{\frac{1}{n}})}{t^{\frac{1}{n}}}=f(\sigma)g(\tau).
\end{equation}

The trivial element of $Gal(K(t^{\frac{1}{n}})/K)$ maps to $1\in \mu_{n}$, therefore, f is an isomorphism. With Kummer Theory, if $K$ is complete under the valuation defined by $t$, we have the following exact sequence:
\[
\begin{CD}
1 @>>>  \mu_{n}  @>>>  Gal(K_{t}^{ab}/K)  @>>>  Gal(F^{ab}/F)  @>>>  1 \\
\end{CD}
\]
where $K_{t}^{ab}$ means the maximal abelian tamely ramified extension field of $K$.\\

Next, let $R$ be a two dimensional local ring which is a completion of $\mathcal{O}_{K}[x, y] = \mathcal{O}_{K}[X, Y]/(XY-\pi)$ at a maximal ideal $(x, y)$, where $\mathcal{O}_{K}$ is an integer ring of a finite extension field $K$ of $\mathbb{Q}_{p}$ or $\mathbb{F}_{p}((t))$ which contains $\pi^{1/m}$, and $\pi=p$ if $\mathbb{Q}_{p}\subset \mathcal{O}_{K}$, $\pi=t$ if $\mathbb{F}_{p}((t))\subset \mathcal{O}_{K}$. Further we assume that the order of the group of root of unity of $\mathcal{O}_{K}$ is $m$. Let $\hat{R_{x}}$ (resp.$\hat{R_{y}}$)  be the completion of $R$ at a height 1 prime ideal $(x)$ (resp.$(y)$) of $R$. Then the fractional field $\hat{K_{x}}$ (resp.$\hat{K_{y}}$) of $\hat{R_{x}}$ (resp.$\hat{R_{y}}$) is a 2 dimensional local field in the sense of \cite{Kato2} and \cite{Kato3}. In this case, we must note that $Gal(\hat{K_{x}}(x^{\frac{1}{n}})/\hat{K_{x}})\simeq Gal(\hat{K_{y}}(y^{\frac{1}{n}})/\hat{K_{y}})$ via $\sigma\mapsto \sigma^{-1}$ because of $xy=p$. Then reciprocity map from $K_{2}^{M}\hat{K_{x}}$ to $Gal(\hat{K_{x}^{ab}}/\hat{K_{x}})$ (in the sense of  \cite{Kato2} and \cite{Kato3}) induces the following tame reciprocity map:\\

\[
  \begin{CD}
     1 @>>>  \mu_{n}  @>>>  Gal(\hat{K_{x,t}^{ab}}/\hat{K_{x}})  @>>>  Gal(k(y)^{ab}/k(y))  @>>>  1 \\
    @.     @AA{tame\ symbol}A  @AA{Reciprocity\ Law}A  @AA{Reciprocity\ Law}A   @. \\
     0 @>>>  K_{2}^{M}k(y) @>>>  K_{2}^{M}\hat{K_{x}}/U^{1}K_{2}^{M}\hat{K_{x}} @>>>  k(y)^{\times} @>>>  0
  \end{CD}
\]
Here please note that $k(y)$ is the residue field of $\hat{K_{x}}$ and it is a field of formal power series of $y$ over $F$. The structure of $K_{2}^{M}k(y)$ is well known, it is a direct sum of $¥mu_{n}$ and a divisible group. This shows that profinite completion of $K_{2}^{M}\hat{K_{x}}/U^{1}K_{2}^{M}$ is isomorphic to $Gal(\hat{K_{x,t}^{ab}}/\hat{K_{x}})$.\\


\section{The Galois group $Gal(K_{R,t}^{ab}/K_{R}$)}
Let $K$ be a finite extension field of $\mathbb{Q}_{p}$ or $\mathbb{F}_{p}((t))$, and $\mathcal{O}_{K}$ be its integer ring. $R$ is a completion of $\mathcal{O}_{K}[x, y] = \mathcal{O}_{K}[X, Y]/(XY-\pi)$ at a maximal ideal $(x, y)$, where $\mathcal{O}_{K}$ is an integer ring of a finite extension field $K$ of $\mathbb{Q}_{p}$ or $\mathbb{F}_{p}((t))$ which contains $\pi^{1/m}$, and $\pi=p$ if $\mathbb{Q}_{p}\subset \mathcal{O}_{K}$, $\pi=t$ if $\mathbb{F}_{p}((t))\subset \mathcal{O}_{K}$. Further we assume that the order of the group of root of unity of $\mathcal{O}_{K}$ is $m$ and it is not divisible by $p$. In this case, since $xy=p$ on the fractional field $K_{R}$ of $R$ and $x^{1/m}y^{1/m}=\pi^{1/m}\in \mathcal{O}_{K}$, $K_{R}(x^{1/m}) =  K_{R}(y^{1/m})$.\\

Note that $K_{R}(x^{1/m}) =  K_{R}(y^{1/m})/K_{R}$ is unramified at all height one prime ideals of $R$ other than $(x)$ and $(y)$. Though the author omits the proof, the readers can understand this from logarithmic structure, see Theorem 3.5 of \cite{Kato1} for more details.\\

Therefore we have the following Theorem.
\begin{thm}
Let $K$ be a finite extension field of $\mathbb{Q}_{p}$ or $\mathbb{F}_{p}((t))$, and $\mathcal{O}_{K}$ be its integer ring. $R$ is a completion of $\mathcal{O}_{K}[x, y] = \mathcal{O}_{K}[X, Y]/(XY-\pi)$ at a maximal ideal $(x, y)$, where $\mathcal{O}_{K}$ is an integer ring of a finite extension field $K$ of $\mathbb{Q}_{p}$ or $\mathbb{F}_{p}((t))$ which contains $\pi^{1/m}$, and $\pi=p$ if $\mathbb{Q}_{p}\subset \mathcal{O}_{K}$, $\pi=t$ if $\mathbb{F}_{p}((t))\subset \mathcal{O}_{K}$. Then there exists an exact sequence on $Gal(K_{R,t}^{ab}/K_{R})$:

\[
\begin{CD}
1 @>>> F^{\times}  @>>>  Gal(K_{R,t}^{ab}/K_{R})  @>>>  \hat{\mathbb{Z}}  @>>>  0. \\
\end{CD}
\]
where $F$ is the residue finite field of $R$ and $K_{R,t}^{ab}/K_{R}$ is the maximal abelian extension field whixch is unramified at all height one prime ideals of $R$ other than $(x)$ and $(y)$ and tamely ramified at $(x)$ and $(y)$.
\end{thm}


\section{The Main Theorem}

\begin{thm}[\bf{Tamely ramified class field theory of two dimensional local ring}]
Let $K$ be a finite extension field of $\mathbb{Q}_{p}$ or $\mathbb{F}_{p}((t))$, and $\mathcal{O}_{K}$ be its integer ring. $R$ is a completion of $\mathcal{O}_{K}[x, y] = \mathcal{O}_{K}[X, Y]/(XY-\pi)$ at a maximal ideal $(x, y)$, where $\mathcal{O}_{K}$ is an integer ring of a finite extension field $K$ of $\mathbb{Q}_{p}$ or $\mathbb{F}_{p}((t))$ which contains $\pi^{1/m}$, and $\pi=p$ if $\mathbb{Q}_{p}\subset \mathcal{O}_{K}$, $\pi=t$ if $\mathbb{F}_{p}((t))\subset \mathcal{O}_{K}$. Then there exists an injective Reciprocity Map $SK_{1}^{log}(R)\longrightarrow Gal(K_{R,t}^{ab}/K_{R})$ and the profinite completion of $SK_{1}^{log}(R)$ is isomorphic to $Gal(K_{R,t}^{ab}/K_{R})$. Here, $K_{R,t}^{ab}$ is the maximal abelian extension field of $K_{R}$ which is tamely ramified at $x$ and $y$, and unramified at all height 1 prime ideals of $R$ other than $x$ and $y$.
\end{thm}

\begin{proof}
For $x$ and $y$, we apply the reciprocity maps defined in the previous section. On the other hand, because of the relation $xy=p$, we must take care the following relation:

$$
\begin{array}{ccc}
Gal(\hat{K}_{R,x}(x^{\frac{1}{n}})/\hat{K}_{R,x}) & \stackrel{inverse}{\simeq} & Gal(\hat{K}_{R,y}(y^{\frac{1}{n}})/\hat{K}_{R,y}) \subset Gal(K_{R,t}^{ab}/K_{R}) \\
\rotatebox{90}{$\in$} & & \rotatebox{90}{$\in$} \\
\sigma & \longmapsto & \sigma^{-1}
\end{array}
$$

where, $\hat{K}_{R,x}$ (resp.$\hat{K}_{R,y}$) is the fractional field of a completion ring of $R$ at $(x)$ (resp.$(y)$). This fact and Theorem (0.5) of \cite{Saito} mean that any element of $K_{2}^{M}K_{R}$ is $0$ in $Gal(K_{R,t}^{ab}/K_{R})$ through the composite of the following 2 maps:

\begin{equation}
K_{2}^{M}K_{R} \to (\bigoplus_{\mathfrak{p}\in P} k(\mathfrak{p})^{\times}) \oplus (K_{2}^{M}\hat{K}_{R,x}/U^{1}K_{2}^{M}\hat{K}_{R,x}) \oplus (K_{2}^{M}\hat{K}_{R,y}/U^{1}K_{2}^{M}\hat{K}_{R,y})
\end{equation}

\begin{equation}
\bigoplus_{\mathfrak{p}\in P} Gal(k(\mathfrak{p})^{ab}/k(\mathfrak{p})) \oplus Gal(\hat{K}_{R,x,t}^{ab}/\hat{K}_{R,x}) \oplus Gal(\hat{K}_{R,y,t}^{ab}/\hat{K}_{R,y}) \to Gal(K_{R,t}^{ab}/K_{R}) 
\end{equation}
where $\hat{K}_{R,x,t}^{ab}$ (resp.$\hat{K}_{R,y,t}^{ab}$) is the maximal tamely ramified abelian extension field of $\hat{K}_{R,x}$ (resp.$\hat{K}_{R,y}$), and the map of (4.2) is the composite of natural maps of Galois groups. Therefore, it induces the following map and the profinite completion of $SK_{1}^{log}(R)$ is isomorphic to $Gal(K_{R,t}^{ab}/K_{R})$ because of Proposition 2.1.
\begin{equation}
\psi^{log}: SK_{1}^{log}(R)\rightarrow Gal(K_{R,t}^{ab}/K_{R}).
\end{equation}
\end{proof}

Note that the following sequence is exact and $F$ is the residue field of $K$.
\[
\begin{CD}
0 @>>> F^{\times}  @>>>  \hat{SK_{1}^{log}(R)}  @>>>  \hat{\mathbb{Z}}  @>>>  0. \\
\end{CD}
\]



\end{document}